\documentclass[letterpaper, 10 pt, conference]{ieeeconf}  

\IEEEoverridecommandlockouts                              

\usepackage{amssymb}
\usepackage{xfrac}
\usepackage{bigints}
\usepackage{relsize}
\usepackage[letterpaper,top=54pt,bottom=58pt,left=54pt,right=54pt]{geometry}
\usepackage{amsmath, amssymb, amsfonts, amsthm}
\usepackage[makeroom]{cancel}
\usepackage{graphics} 
\usepackage{epsfig} 
\usepackage{mathptmx} 
\usepackage{times} 
\usepackage{amsmath} 
\usepackage{amssymb}  
\usepackage{xcolor}
\usepackage{dsfont}

\usepackage{subcaption}
\usepackage{hyperref}

\makeatletter
\DeclareOldFontCommand{\rm}{\normalfont\rmfamily}{\mathrm}
\DeclareOldFontCommand{\sf}{\normalfont\sffamily}{\mathsf}
\DeclareOldFontCommand{\tt}{\normalfont\ttfamily}{\mathtt}
\DeclareOldFontCommand{\bf}{\normalfont\bfseries}{\mathbf}
\DeclareOldFontCommand{\it}{\normalfont\itshape}{\mathit}
\DeclareOldFontCommand{\sl}{\normalfont\slshape}{\@nomath\sl}
\DeclareOldFontCommand{\sc}{\normalfont\scshape}{\@nomath\sc}
\makeatother

\usepackage{csquotes}  
\newcommand\q{\enquote}

\usepackage{physics}   

\newcommand \Limsup {\mathop{\overline{\lim}}}

\newcommand \N   {\mathbb{N}}
\newcommand \R   {\mathbb{R}}

\newcommand \Kinf{\mathcal{K_\infty}}

\newcommand \KL  {\mathcal{KL}}

\newcommand{\vertiii}[1]{{\left\vert\kern-0.25ex\left\vert\kern-0.25ex\left\vert #1 
    \right\vert\kern-0.25ex\right\vert\kern-0.25ex\right\vert}}

\newcommand \id  {\operatorname{id}}

\newcommand{\normt}[1]{{\left\vert\kern-0.25ex\left\vert\kern-0.25ex\left\vert #1 
		\right\vert\kern-0.25ex\right\vert\kern-0.25ex\right\vert}}

\newcommand{\mir}[1]{{\color{red}\bf AM: #1}}     

\newif\ifMath					
\newif\ifEngi					

\newif\ifDFGtext					 

\newif\ifAndo              
													
\newif\ifExercises					
\newif\ifSolutions          
\newif\ifGerman							
\newif\ifEnglish						

\newif\ifnothabil						

\newif\ifFuture							

\newif\ifConf                    
\newif\ifJournal								 

\newif\ifNOTFORBOOK
\newif\ifFullVersion
\newif\ifExludedDueToSpaceReasons

\usepackage{xifthen}

\newcommand{\einsnorm}[2]{\ensuremath{
    \!\!\;\!\!\!\;
    \left\bracevert\!\!\!\!\!\left\bracevert
    \!
		\ifthenelse{\isempty{#2}}{#1}{#1(#2)}
    \!
      \right\bracevert\!\!\!\!\!\right\bracevert
    \!\!\;\!\!\!\;
  }}

\usepackage{xcolor}
\definecolor{blond}{rgb}{0.98, 0.94, 0.75}
	
\newlength\mytemplen
\newsavebox\mytempbox

\makeatletter
\newcommand\mybluebox{%
    \@ifnextchar[
       {\@mybluebox}%
       {\@mybluebox[0pt]}}

\def\@mybluebox[#1]{%
    \@ifnextchar[
       {\@@mybluebox[#1]}%
       {\@@mybluebox[#1][0pt]}}

\def\@@mybluebox[#1][#2]#3{
    \sbox\mytempbox{#3}%
    \mytemplen\ht\mytempbox
    \advance\mytemplen #1\relax
    \ht\mytempbox\mytemplen
    \mytemplen\dp\mytempbox
    \advance\mytemplen #2\relax
    \dp\mytempbox\mytemplen
    \colorbox{blond}{\hspace{1em}\usebox{\mytempbox}\hspace{1em}}}

\makeatother

\makeatletter
\let\origd=\d
\renewcommand*\d{
  \relax\ifmmode
    \mathrm{d}%
  \else
    \expandafter\origd
  \fi
}\makeatother

\usepackage{mathtools}

\makeatletter 
\newcommand{\pushright}[1]{\ifmeasuring@#1\else\omit\hfill$\displaystyle#1$\fi\ignorespaces}
\newcommand{\pushleft}[1]{\ifmeasuring@#1\else\omit$\displaystyle#1$\hfill\fi\ignorespaces}
\makeatother 

\newcounter{syscounter}

\newcounter{WPcounter}
\newcounter{PRcounter}

\newtheorem{remark}{Remark}
\newtheorem{theorem}{Theorem}

\newtheorem{corollary}{Corollary}
\newtheorem{definition}{Definition}
\newtheorem{assumption}{Assumption}

\newtheorem{prop}{Proposition}

\newtheorem*{goal*}{Goal}

\title{\LARGE \bf
Sampled-data and event-triggered control of globally Lipschitz infinite-dimensional systems
}

\author{Rami Katz$^{1}$ and Andrii Mironchenko$^{2}$
\thanks{The work of RK was supported by the European Union through the ERC INSPIRE Grant under Project 101076926. AM has been supported by the Heisenberg grant  MI 1886/3-1 of the German Research Foundation.}
\thanks{$^{1}$R. Katz is with the Department of Industrial Engineering,
        Trento University, Via Sommarive, 9 - 38123 Povo, Italy.
        {\tt\small ramkatsee@gmail.com}}%
\thanks{$^{2}$ A. Mironchenko is with Universit\"at Bayreuth, Mathematisches Institut, Universit\"atsstra\ss e 30, Bayreuth, Germany.
        {\tt\small andersmir@gmail.com}}%
}

\begin{document}

\maketitle
\thispagestyle{empty}
\pagestyle{empty}

\begin{abstract}
We show that if a linear infinite-dimensional system is exponentially stabilizable by compact feedback, it is also stabilizable by means of a sampled-data feedback that is fed through a globally Lipschitz nonlinearity, provided that the sector bound for the nonlinearity and the sampling time is small enough. 
Next we develop a switching-based event-triggered control scheme stabilizing the system with a reduced number of switching events.
We demonstrate our results on an example of finite-dimensional stabilization of a Sturm-Liouville parabolic system.
\end{abstract}

\section{Introduction}

Event-based control is a paradigm that allows for a rigorous digital implementation of continuous-time controllers as well as
for the efficient use of communications, computational, and actuating resources by updating control
inputs only at time instances when it is needed. Due to these important features, a lot of attention has been devoted to event-based control of nonlinear systems \cite{Tab07,PTN15, BoH14,selivanov2015event11} and infinite-dimensional systems \cite{EGM18, KFS21, EKK21, RDE21}.

In the seminal paper \cite{Tab07}, a general event-triggering control scheme for nonlinear ODE systems has been proposed, provided that the system can be input-to-state stabilized by means of classical controllers. 
In the same work, the absence of Zeno behaviour for the sequence of the switches has been shown. 
The finite-dimensional arguments from \cite{Tab07} cannot, however, be directly translated to the infinite-dimensional setting due to the unboundedness of the generators of strongly continuous semigroups. 
Tabuada's approach presumes not only input-to-state stability (ISS) of a closed-loop system but also the existence of a corresponding coercive ISS Lyapunov function. 
For nonlinear ODE systems, such an ISS Lyapunov function always exists if the nonlinearity is Lipschitz continuous \cite{SoW95}. For infinite-dimensional systems, the existence of an ISS Lyapunov function has been shown for evolution equations with Lipschitz nonlinearities in \cite{MiW17c}. However, in contrast to ODE systems, for evolution equations in Banach spaces, the assumption of Lipschitz continuity is quite restrictive as linear systems with unbounded operators do not fall into this class. Significant problems with the development of the coercive converse Lyapunov results for linear systems with admissible input operators have been reported in \cite{JMP20}, where non-coercive ISS Lyapunov theory has been developed for distributed parameter systems. 
Even though converse ISS Lyapunov functions exist for some classes of infinite-dimensional linear systems with boundary inputs, such as hyperbolic systems of the first order (based on constructions from \cite{BaC16}), parabolic systems with Neumann or Robin boundary conditions \cite{ZhZ18}, and for certain classes of analytic systems \cite{MiS25}, the problem of obtaining general converse coercive ISS Lyapunov theorem for linear systems with unbounded input operators remains unresolved, in spite of the recent progress in the infinite-dimensional ISS theory \cite{MiP20}. 
Hence, event-based controllers are developed for PDEs predominantly in a case-by-case manner, see \cite{RDE21,KBT22,EKK21}.

Our main motivation is the paper \cite{LRT03}, where it was shown that linear infinite-dimensional systems, which are exponentially stabilizable by means of a compact feedback, can also be stabilized by zero-order-hold sample-data controller if the input operator is bounded or if the semigroup is analytic. This result is important as it does not presume a specific structure from the system. 

It was shown in \cite{LRT03} by a simple example that compactness of the feedback is necessary for the validity of sample-data schemes even if the input operator is bounded. 
The compactness requirement sets the limits on the applicability of the event-triggered control. Indeed, in \cite[Theorem 2]{Gib80}, it was shown that a strongly stable contraction semigroup on a Hilbert space with a generator $A$ that is not exponentially stable cannot be exponentially stabilized by means of a compact feedback, i.e. $A+B$ does not generate an exponentially stable semigroup for any compact $B$. Note that compact operators can be represented as the limits of sequences of finite rank operators \cite[Theorem 4.4, p. 41]{Con90}, and by \cite[Theorem 8.1.6, p. 349]{CuZ20}, exponential stabilization by finite-rank operators is possible only if the operator $A$ has finitely many unstable eigenvalues.  


\textbf{Contribution.} 
In this work, we present a general methodology for event-based stabilization of linear systems with globally Lipschitz continuous input operators. Infinite dimensional Lur'e-type systems similar to those considered here have been studied in, e.g., \cite{guiver2019infinite111,gilmore2020stability111,gilmore2020infinite1121,logemann2014infinite}. However, to the best of our knowledge, a general approach for event-based stabilization of such systems has not been considered in the literature so far.

        In Section~\ref{sec:Sampled-data stabilization}, we show that the sample-data controller achieves exponential stabilization of linear systems by means of globally Lipschitz controllers provided that the nonlinearity satisfies the sector condition with the small enough sector bound.
Next, in Section~\ref{sec:Event-based stabilization}, we demonstrate that the switching-based event-trigger control scheme, first proposed in \cite{selivanov2015event11}, can be effectively employed in \eqref{eq:AbstractSys} with \eqref{eq:SampDat} to reduce the number of controller updates. 
For this latter goal, we use the converse ISS Lyapunov results for evolution equations shown in \cite{MiW17c}.

In this paper, we concentrate on the case where the input operators are bounded. More complex case when the input operator is unbounded and the semigroup is analytic will be treated in the future works.





\smallskip
\textbf{Notation.} 
By $\N$ we denote the set of positive integers; $\mathbb{N}_0:=\N\cup\{0\}$. The set of nonnegative reals we denote by $\R_+$. The Euclidean norm on $\R^n$ is denoted by $|\cdot|$. Given a Hilbert space $H$, the standard norm on $H$ is denoted by $\left\|\cdot \right\|_H$. The space of bounded linear operators from $U$ to $H$ is denoted by $\mathcal{B}(U,H)$ with the corresponding norm $\left\| \cdot\right\|_{\mathcal{B}(U,H)}$. For short, we write $\mathcal{B}(H):=\mathcal{B}(H,H)$. For an unbounded linear operator $A$ on $H$ we denote by $\operatorname{dom}(A)$ the domain of $A$. By $L_2(0,L)$ we denote the Hilbert space of scalar and square integrable functions on the interval $(0,L)$. A function $f:[0,a)\to \R_+$ belongs to $\mathcal{K}$ if $f(0)=0$ and $f$ is continuous and strictly increasing. $f$ belongs to $\mathcal{K}_{\infty}$ if, in addition, $a=\infty$ and $\lim_{r\to \infty}f(r)=\infty$. A continuous function $g:\R_+^2\to \mathbb{R}$ belongs to $\mathcal{K}\mathcal{L}$ if $g(\cdot,\ell)\in \mathcal{K}$ for each $\ell \in\R_+$ and  for each $r\in\R_+$, $g(r,\cdot)$ is decreasing and tends to $0$ at $\infty$.



\section{Problem formulation}

We consider the following system with state space $H$ and input space $U$, which are both Hilbert spaces
\begin{equation}\label{eq:AbstractSys}
\begin{array}{lll}
\dot{x}(t) = Ax(t)+Bf\left(u(t) \right),\quad x(0)=x_0.
\end{array}
\end{equation}
Here $B\in \mathcal{B}(U,H)$ and  $A:\operatorname{dom}(A)\subset H \to H$ is the infinitesimal generator of a $C_0$-semigroup $\left\{T(t)\right\}_{t\geq 0}$ on $H$ with 
\begin{center}
$\left\|T(t)\right\|_{\mathcal{B}(H)}\leq Me^{\nu t},\quad t\geq 0$ for $M\geq 1$ and $\nu>0$.    
\end{center}
Next we assume that the the pair $(A,B)$ is exponentially stabilizable via compact feedback. 
\begin{assumption}\label{assum:StabilizingFeed}
There exists a compact operator $F\in \mathcal{B}(H,U)$ such that the operator $A+BF$ generates an exponentially stable $C_0$-semigroup on $H$ denoted by $\left\{T_{BF}(t) \right\}_{t\geq 0}$. 
That is, there exist $N\geq 1$ and $\xi>0$ such that 
\begin{equation}\label{eq:ExpStabTbf}
\begin{array}{lll}
\left\|T_{BF}(t) \right\|_{\mathcal{B}(H)} \leq Ne^{-\xi t},\ t\geq 0.
\end{array}
\end{equation}
\end{assumption}
Subject to Assumption \ref{assum:StabilizingFeed}, the authors of \cite{LRT03} showed that for a sufficiently small $\tau^*>0$, and any $\tau \in (0,\tau^*)$, the sampled-data feedback controller
\begin{equation}\label{eq:SampDat}
\begin{array}{lll}
u(t) = Fx(k\tau),\quad t\in [k\tau,(k+1)\tau),\ k\in \mathbb{N}
\end{array}
\end{equation}
achieves exponential stabilization of the system 
\[
\dot{x}(t) = Ax(t)+Bu(t).
\]

For the function $f$, we make the following assumption:
\begin{assumption}\label{eq:AssumpLipsch}
$f\circ F:H\to U$ is globally Lipschitz continuous. Moreover, 
$f\circ F$ satisfies the following \emph{sector condition}
\begin{equation}\label{eq:SectorCond}
\begin{array}{lll}
\left\|f(Fx)-Fx \right\|_{U}\leq \theta_{f}\left\|x \right\|_{H},\qquad x\in H,
\end{array}
\end{equation}
for a constant $\theta_f>0$. Note that \eqref{eq:SectorCond} implies that $f(0)=0$.
\end{assumption}


\begin{remark}
By a suitable modification of the arguments presented in this paper, one can obtain analogous results for sampled-data and ET stabilization of semilinear systems 
\begin{equation}
\label{eq:Nonlin-Sys}
\begin{array}{lll}
\dot{x}(t) = Ax(t)+ f(x(t)) +Bu(t),\quad x(0)=x_0
\end{array}
\end{equation}
by means of linear controllers, and where $f:H\to H$ is continuously differentiable and satisfies the sector condition $\left\|f(x) \right\|_H\leq \theta_f\left\|x \right\|_H$ for all $x\in H$ with $\theta_f>0$ small enough. Due to space limitations, we omit the explicit details.
\end{remark}

\section{Sampled-data stabilization of \eqref{eq:AbstractSys}}
\label{sec:Sampled-data stabilization}

Motivated by \cite{LRT03}, in this section we consider \eqref{eq:AbstractSys} and show that for small enough $\theta_f$ and $\tau>0$ in \eqref{eq:SectorCond}, the sampled-data feedback controller \eqref{eq:SampDat}, achieves exponential stabilization of the closed-loop system.
We will exploit it for the development of the event-based schemes in Section~\ref{sec:Event-based stabilization}.

Pick any $\tau>0$. 
With the notation $x_k:=x(k\tau)$, $k\in\N$, we can write down the dynamics of the closed-loop system \eqref{eq:AbstractSys} with a sample-data controller \eqref{eq:SampDat} at time moments $\left\{\tau k \right\}_{k=0}^{\infty}$ as the discrete-time system 

\begin{equation}\label{eq:DiscLip}
\begin{array}{lll}
&x_{k+1} = \Psi_{\tau}(x_k), 
\end{array}
\end{equation}
where for any $x \in H$
\begin{align}
\Psi_{\tau}(x) 
&:=  T(\tau)x+\int_0^{\tau}T(s)Bf(Fx) \mathrm{d}s.
\end{align}
We say that $\Psi_{\tau}$ is \emph{power stable} if there exist $L\in \mathbb{N}$ and $q\in (0,1)$ such that for all $x_0\in H$, the solution $\left\{x_k \right\}_{k\in \mathbb{N}}$ of \eqref{eq:DiscLip} satisfies
\begin{equation}\label{eq:PowerStab}
\begin{array}{lll}
\left\|x_k \right\|_H\leq L q^k\left\|x_0\right\|_H.
\end{array}
\end{equation}
In other words, $\Psi_{\tau}$ is power stable if the system \eqref{eq:DiscLip} if globally exponentially stable.

The following proposition guarantees that $\Psi_{\tau}$ is power stable provided that $\tau$ and $\theta_{f}$ are small enough.

\begin{prop}
\label{Prop:NonlinPowStab}
Let the constant $\theta_f$ satisfy the condition
\begin{equation}\label{eq:PowerStabCond}
\begin{array}{lll}
\theta_f<\xi\beta^{-1},\quad 
\beta := MN \left\|B\right\|_{\mathcal{B}(U,H)}.
\end{array}
\end{equation}
Then, there exists $\tau_*>0$ such that that $\Psi_{\tau}$ is power stable for all $\tau\in (0,\tau_*)$.
\end{prop}

\begin{proof}
We represent $\Psi_{\tau}$ as $\Psi_{\tau} := \Delta_{\tau} + \Phi_{\tau}$, with
\begin{equation*}
\begin{array}{lll}
&\Delta_{\tau}(x) := T(\tau)x+\int_0^{\tau}T(s)BFx \mathrm{d}s,\vspace{0.1cm} \\
&\Phi_{\tau}(x) := \int_0^{\tau}T(s)B\left[f\left(Fx \right)-Fx \right]\mathrm{d}s.
\end{array}
\end{equation*}


Let us take $x,y\in H$. Then,
\begin{equation}\label{eq:PhiLip}
\begin{array}{lll}
\left\|\Phi_{\tau}(x) \right\|_H &\le \int_0^{\tau} \left\|T(s) B \left[f\left(Fx \right) - Fx \right] \right\|_H \mathrm{d}s\\
&\leq M \left\|B \right\|_{\mathcal{B}\left(U,H \right)}\left\|f\left(Fx \right)-Fx \right\|_{H}\int_0^{\tau}e^{\nu s}\mathrm{d}s\\
&\leq \theta_f M\left\|B \right\|_{\mathcal{B}\left(U,H \right)}\frac{e^{\nu \tau}-1}{\nu}\left\|x \right\|_H .
\end{array}
\end{equation}

Consider $\Delta_{\tau}$ and the following equivalent norm on $H$
\begin{equation}\label{eq:EquivNorm}
\left[\left[x  \right] \right]_H = \sup_{t\geq 0} \left\|e^{\xi t}T_{BF}(t)x \right\|_H,
\end{equation}
which satisfies $\left\|x \right\|_H\leq \left[\left[x \right] \right]_H\leq N \left\|x \right\|_H$ for all $x\in H$. Then, \cite[Theorem 3.1]{LRT03} implies that 
\begin{equation}\label{eq:PowerStabDelta}
\begin{array}{lll}
\left[ \left[\Delta_{\tau} \right] \right]_{\mathcal{B}(H)}\leq 1-\xi\tau+o(\tau),\quad \tau\to 0^+.
\end{array}
\end{equation}
On the other hand, from \eqref{eq:PhiLip} we obtain
\begin{equation}\label{eq:PhiLip1}
\begin{array}{lll}
\left[ \left[\Phi_{\tau}(x) \right] \right]_H \leq \theta_{f}\beta \frac{e^{\nu t}-1}{\nu} \left\|x  \right\|_H 
\leq \theta_{f}\beta \frac{e^{\nu t}-1}{\nu} \left[\left[x \right] \right]_H.
\end{array}
\end{equation}
Hence, combining \eqref{eq:PowerStabDelta} and \eqref{eq:PhiLip1}, we have 
\begin{equation}\label{eq:PsiPointwise}
 \begin{array}{lll} 
\left[\left[\Psi_{\tau}(x) \right] \right]_H &\leq \left[\left[\Delta_{\tau}x \right] \right]_H+ \left[\left[\Phi_{\tau}(x) \right] \right]_H\\
&\leq \left(1-\xi \tau +\theta_f \beta \tau +o(\tau) \right)\left[\left[x \right] \right]_H,
 \end{array}
\end{equation}
as $\tau \to 0^+$. Therefore, condition \eqref{eq:PowerStabCond} guarantees that there exists $\tau_*>0$ such that for $\tau \in (0,\tau_*)$ and all $x\in H$, we have $\left[\left[\Psi_{\tau}(x) \right] \right]_H\leq q \left[ \left[x \right] \right]_H$,
where $q\in (0,1)$. Recalling \eqref{eq:DiscLip} and returning to the original norm yields
\begin{equation*}
\begin{array}{lll}
\left\|x_k \right\|_H \leq \left[\left[ x_k\right] \right]_H \leq q^k \left[\left[x_0 \right] \right]_H \leq N q^k \left\|x_0 \right\|_H. \hfill \blacksquare
\end{array}
\end{equation*}
\end{proof}
Proposition \ref{Prop:NonlinPowStab} implies that for $\tau \in (0,\tau_*)$, the sampled-data feedback controller \eqref{eq:SampDat} exponentially stabilizes \eqref{eq:AbstractSys}.
\begin{corollary}\label{Cor:StabSamp}
Under the same condition as Proposition \ref{Prop:NonlinPowStab} and $\tau \in (0,\tau_*)$, the closed-loop system \eqref{eq:AbstractSys} with controller \eqref{eq:SampDat} is exponentially stable, i.e. there exist $G\geq 1$,$\chi>0$ such that
\begin{align}
\label{eq:ExpStab-SD-control}
\left\| x(t)\right\|_H\leq Ge^{-\chi t}\left\|x_0 \right\|_H,\quad \forall t\geq 0.     
\end{align}
\end{corollary}
\begin{proof}
For $t\geq 0$, write $t=k\tau+ \ell\in [k\tau, (k+1)\tau)$. Then,
\begin{equation*}
\begin{array}{lll}
x(t) = \Delta_{\ell}x(k\tau)+\Phi_{\ell}(x(k\tau)).
\end{array}
\end{equation*}
Since $\left\| \Delta_{\ell}\right\|_H\leq 1$  and $\left\|\Phi_{\ell}(x) \right\|_H\leq A \left\|x \right\|_H$ for some constant $A>0$ and all $\ell\in [0,\tau]$, we conclude that 
\begin{equation*}
\begin{array}{lll}
\left\|x(t) \right\|_H&\leq \left(A+1 \right)\left\|x(k\tau) \right\|_{H}\leq (A+1)N q^k\left\|x_0 \right\|_H\\
& \leq (A+1)Ne^{-\ln(q)}e^{\frac{\ln(q)}{\tau}t} \left\|x_0 \right\|_H. \hfill \blacksquare
\end{array}
\end{equation*}
\end{proof}

\section{Event-based stabilization of \eqref{eq:AbstractSys}}
\label{sec:Event-based stabilization}

\subsection{Disturbed closed-loop system}

In the next section, we study the closed-loop \eqref{eq:AbstractSys} with a controller \eqref{eq:SampDat} subject to \eqref{eq:PowerStabCond} and a switching-based event-trigger control scheme. 
Our analysis relies on a suitable converse Lyapunov theorem.

Let
$\mathfrak{D}$ denote the set of disturbances, containing all $d:\R_+\to H$ which are locally Lipschitz continuous. We begin by considering the following system
\begin{equation}\label{eq:ClosedLoopCont}
\begin{array}{lll}
\dot{x}(t) = Ax(t)+Bf\left(Fx(t) \right)+d(t),\quad x(0)=x_0,
\end{array}
\end{equation}
where $d\in \mathfrak{D}$ and the sector-bound parameter of $f$ satisfies \eqref{eq:PowerStabCond}. 
As the nonlinearity in \eqref{eq:ClosedLoopCont} is globally Lipschitz, the mild solution of \eqref{eq:ClosedLoopCont} corresponding to the initial condition $x_0$ and input $d$ exists on $\R_+$ and is unique.
We will henceforth denote by $\phi(t;x_0,d)$ the flow induced by \eqref{eq:ClosedLoopCont}.

For a continuous function $y:\R \to \R$, define 
the right upper Dini derivative by $D^+y(t):=\Limsup_{h \to +0}\frac{y(t+h)-y(t)}{h}$.

\begin{definition}\label{def:ISSLyap}
Consider the system \eqref{eq:ClosedLoopCont}, where $d\in \mathfrak{D}$. We say that a continuous function $V:H\to \R_+$ is a \emph{coercive ISS Lyapunov function} for \eqref{eq:ClosedLoopCont} if the following conditions hold:
\begin{enumerate}
    \item There exist $\psi_1,\psi_2 \in \mathcal{K}_{\infty}$ such that for all $x\in H$
    \begin{equation*}
    \begin{array}{lll}
    \psi_1\left(\left\| x\right\|_H \right)\leq V(x) \leq  \psi_2\left(\left\| x\right\|_H \right).
    \end{array}
    \end{equation*}
    \item There exist $\alpha,\gamma \in \mathcal{K}_{\infty}$ such that for all $x\in H$ and $d\in \mathfrak{D}$
    \begin{align*}
    \hspace{-8mm}\dot{V}_d(x):= D^+V(\phi(t,x,d))|_{t=0}
    &=\limsup_{h\to 0^+}\frac{V(\phi(h;x,d))-V(x)}{h}\\
    &\leq -\alpha\left(V(x) \right)+ \gamma \left(\left\|d(0) \right\|_H \right).
    \end{align*}   
\end{enumerate}
\end{definition}

The next \emph{converse ISS Lyapunov theorem} is of independent interest.
\begin{prop}\label{prop:ConvISS}
Consider \eqref{eq:ClosedLoopCont} with $d\in \mathfrak{D}$ and subject to Assumptions \ref{assum:StabilizingFeed},\ref{eq:AssumpLipsch}. Let $\theta_f>0$ satisfy $\theta_f<\xi M\beta^{-1}$ and choose $\zeta \in (\theta_f\beta M^{-1},\xi)$. Let the function 
$V:H\to \R_+$, given by 
\begin{equation}\label{eq:Vdef}
\begin{array}{lll}
V(x) := \sup_{t\geq 0}\big\|e^{\zeta t}T_{BF}(t)x \big\|_H,\qquad x \in H.
\end{array}
\end{equation}
Then, $V$ is a coercive and globally Lipschitz ISS Lyapunov functional for \eqref{eq:ClosedLoopCont}. 
In particular, for any $d\in \mathfrak{D}$ and $x\in H$, we have the dissipation inequality 
\begin{equation}\label{eq:DissipIneq}
\begin{array}{lll}
\dot{V}_d(x)
&\leq -\left(\zeta-\theta_f \beta M^{-1} \right)V(x)+N\left\|d(0) \right\|_{H}.
\end{array}
\end{equation}
\end{prop}
\begin{proof}
Coercivity and global Lipschitzness of $V$ follow from Assumption \ref{assum:StabilizingFeed} and \cite[Proposition 7]{MiW17c}, where $V$ was considered for \emph{linear} systems.

Consider now $x_0\in \operatorname{dom}(A)$ and denote 
\begin{equation*}
h(t,x) :=Bf(Fx)+d(t),\qquad t\geq 0, \ x \in H.
\end{equation*}
Fixing $T>0$, since $d$ is Lipschitz continuous on $[0,T]$, Assumption \ref{eq:SectorCond} implies that $h:[0,T]\times H\to H$ is Lipschitz continuous in both variables. Since $H$ is reflexive, \cite[Theorem 6.1.6]{pazy1983semigroups} shows that the mild solution $x(t)$ of \eqref{eq:ClosedLoopCont} that exists and is unique due to Lipschitz continuity of the nonlinearity in \eqref{eq:ClosedLoopCont}, is a strong solution, meaning that $x(t)$ is differentiable almost (a.e.) everywhere, $\dot{x}$ is in $L^1_{loc}([0,\infty],H)$ and a.e. on $\R_+$ it holds that
\begin{equation}
\begin{array}{lll}
\dot{x}(t) = \left(A+BF \right)x(t)+B\left(f(Fx(t))-Fx(t) \right)+ d(t).
\end{array}
\end{equation}
Therefore, by \cite[Chapter 6.1]{pazy1983semigroups}, the solution $x(t)=\phi(t;x_0,d)$ to \eqref{eq:ClosedLoopCont} with $x_0\in \operatorname{dom}(A)$ can be represented as
\begin{align}\label{eq:StrongSol}
x(t)&= T_{BF}(t)x_0+\int_0^{t}T_{BF}(t-s)d(s) \mathrm{d}s \nonumber\\
&\hspace{12mm}+ \int_0^tT_{BF}(t-s) B\left(f(Fx(s))-Fx(s) \right)\mathrm{d}s.
\end{align}
Now, given $h>0$, $x\in H$ and $d\in \mathfrak{D}$, we use the fact that $V$ is an equivalent norm on $H$ to obtain: 
\begin{align*}
h^{-1}&\Big(V\big(\phi(h;x,d)\big)-V(x) \Big)\\
&\leq \frac{1}{h}\Big( V(T_{BF}(h)x)-V(x)\Big)+V\Big(\frac{1}{h}\int_0^hT_{BF}(h-s)d(s)\mathrm{d}s \Big)\\
&\hspace{10mm}+ V\Big(\frac{1}{h}\int_0^hT_{BF}(h-s) B\left(f(Fx(s))-Fx(s) \right)\mathrm{d}s \Big).
\end{align*}
We analyse the three terms on the right-hand side separately. First, since $V\left(T_{BF}(h)x \right)\leq e^{-\zeta h}V(x)$, we have 
\begin{equation*}
\begin{array}{lll}
h^{-1}\left( V(T_{BF}(h)x)-V(x)\right)\leq \frac{e^{-\zeta h}-1}{h}V(x) \overset{h\to 0^+}{\longrightarrow} -\zeta V(x).
\end{array}
\end{equation*}
By continuity of $V$ and local Lipschitz continuity of $d$,  
\begin{equation*}
\begin{array}{lll}
&V\left(h^{-1}\int_0^hT_{BF}(h-s)d(s)\mathrm{d}s\right) \overset{h\to 0^+}{\longrightarrow}V(d(0))\leq N\left\|d(0) \right\|_H.
\end{array}
\end{equation*}
Finally, by continuity of $V$ and Lipschitz continuity of $f$
\begin{equation*}
\begin{array}{lll}
&\lim_{h\to 0^+}V\left(h^{-1}\int_0^hT_{BF}(h-s) B\left(f(Fx(s))-Fx(s) \right)\mathrm{d}s \right)\\
&\hspace{3mm} = V\left(B\left(f(Fx)-Fx \right) \right)\leq N\left\|B \right\|_{\mathcal{B(U,H)}} \left\|f(Fx)-Fx \right\|_H\\
&\hspace{3mm} \leq \theta_f \beta M^{-1}\left\|x \right\|_H \leq \theta_f \beta M^{-1} V(x).
\end{array}
\end{equation*}
Overall, we obtain that $\eqref{eq:DissipIneq}$ holds, provided $x_0\in \operatorname{dom}(A)$. In particular, denoting $\ell=\zeta-\theta_f \beta M^{-1}$, we have, by the comparison principle \cite[Appendix A]{Mir23},  that 
\begin{equation}\label{eq:ComparisonPr}
\begin{array}{lll}
V(\phi(t;x,d))\leq e^{-\ell t}V(x)+N\int_0^{t}e^{-\ell(t-s)}\left\|d(s) \right\|_{H}\mathrm{d}s
\end{array}
\end{equation}
In general, let  $T>0$, $x_0\in H$, $d\in \mathfrak{D}$ and $\left\{x_n \right\}_{n=1}^{\infty}\subseteq \operatorname{dom}(A)$ such that $\lim_{n\to \infty}x_n=x_0$. For $t\in[0,T]$, substitute $x=x_n$ in \eqref{eq:ComparisonPr} and take $n\to \infty$.  By continuity of $V$, we have $\lim_{n\to \infty}V(x_n)=V(x_0)$. Since $\phi(\cdot;x_n,d)$ converges in $C([0,T],H)$ to $\phi(\cdot;x_0,d)$ \cite[Theorem 6.1.2]{pazy1983semigroups}, continuity of $V$ implies $\lim_{n\to \infty}V\left(\phi(t;x_n,d) \right) = V(\phi(t;x_0,d))$ for $t\in [0,T]$. Thus, \eqref{eq:ComparisonPr} holds on $t\in [0,T]$ with $x=x_0$. Taking $t\in [0,T)$ and $h>0$ such that $t+h\leq T$, we have 
\begin{align}\label{eq:IntegDissip}
&h^{-1}\left(V(\phi(h;x_0,d)) -V(x_0) \right)\nonumber\\
&\leq h^{-1}\big(e^{-\ell h}-1 \big)V(x_0)+h^{-1}N\int_0^{h}e^{-\ell(h-s)}\left\|d(s) \right\|\mathrm{d}s.
\end{align}
Taking $\limsup_{h\to 0^+}$ on both sides of \eqref{eq:IntegDissip}, we have that \eqref{eq:DissipIneq} holds for $x=x_0$. As $T>0$ is arbitrary, the claim is proved.$\hfill \blacksquare$
\end{proof}

\subsection{Switching-based event-triggered (ET) mechanism}

We are now ready to present the considered ET mechanism. In light of Proposition \ref{prop:ConvISS}, we suppose the following:
\begin{assumption}\label{assum:ExistISSLyap}
There is a coercive ISS-Lyapunov function $V:H\to \R_+$ for \eqref{eq:ClosedLoopCont} with corresponding $\psi_1,\psi_2,\alpha,\gamma \in\Kinf$.
\end{assumption}

By Corollary \ref{Cor:StabSamp}, there exists $\tau_*>0$ such that the system \eqref{eq:AbstractSys} with the controller \eqref{eq:SampDat} with $\tau<\tau_*$ is exponentially stable, i.e., \eqref{eq:ExpStab-SD-control} holds. Fix $\tau<\tau_*$, $\sigma\in(0,1)$, and consider the following \emph{switching-based ET controller} \cite{selivanov2015event11} for the system \eqref{eq:AbstractSys}
\begin{subequations}
\label{eq:ETdef}
\begin{align}
&\hspace{-3mm} u(t) = Fx(t_k),\quad t\in [t_k,t_{k+1}),\quad t_0 = 0,\\
& \hspace{-3mm} t_{k+1} = \inf\left\{t\geq t_k+\tau \ ;\right.\\
&\hspace{5mm}\left. \gamma \left( \left\|B\left[f(Fx(t_k)-f(Fx(t)) \right] \right\|_H\right)
\geq \sigma \alpha(V(x))   \right\},
\end{align}
\end{subequations}
where if for some $k\in \mathbb{N}_0 $, the set defining $t_{k+1}$ is empty, we set $t_{k+1}=\infty$. We henceforth refer to the inequality in \eqref{eq:ETdef} as the \emph{ET condition}.

\begin{remark}
By design, \eqref{eq:ETdef} cannot exhibit Zeno behavior, since $t_{k+1}-t_k\geq \tau$ for all $k\in \mathbb{N}_0$. The intuition behind \eqref{eq:ETdef} is that if $t_{k+1}=t_k+\tau$ for all $k\in \mathbb{N}_0$, then the dynamics of \eqref{eq:AbstractSys} subject to 
\eqref{eq:ETdef} is the same as its behavior subject to \eqref{eq:SampDat}, which yields an exponentially stable closed-loop system.     
\end{remark}

\begin{remark}
\label{rem:Comparison-principle}
According to comparison principle \cite[Proposition A.35]{Mir23}, for $\alpha \in\Kinf$ from Assumption~\ref{assum:ExistISSLyap} there is $\kappa \in \KL$ such that
for each $y \in C(\R_+,\R_+)$ satisfying 
\[
D^+ y(t) \le -(1-\sigma)\alpha\left(y(t) \right),\quad t\geq 0,
\]
it holds that $y(t) \leq \kappa(y(0),t)$.    
\end{remark}


With $G$ and $\chi$ as in \eqref{eq:ExpStab-SD-control}, define
\[
\rho(m,n):=\psi_1^{-1}\circ \kappa\left(\psi_2\left(Ge^{-\chi \tau}m \right),n \right)
\]
and introduce the iterates 
\begin{equation}\label{eq:Iterates}
\begin{array}{lll}
&\hspace{-4mm}\mathcal{R}^{(1)}(m,\left\{n\right\}) = \rho(m,n),\\
&\hspace{-4mm}\mathcal{R}^{(N)}(m,\left\{n_j\right\}_{j=1}^N) = \rho\left( \mathcal{R}^{(N-1)}(m,\left\{n_j\right\}_{j=1}^{N-1}),n_N \right),\ N\geq 2.
\end{array}
\end{equation}

We can now formulate our main result
\begin{theorem}\label{thm:MainThm}
If there exist $\mu\in \mathcal{K}\mathcal{L}$ and $\omega\in \mathcal{L}$ such that 
\begin{equation}\label{eq:ETCond}
\begin{array}{lll}
\mathcal{R}^{(N)}(m,\left\{n_j\right\}_{j=1}^N) \leq \mu\left(\omega(N)m,\max_{1\leq j \leq N}n_j \right),
\end{array}
\end{equation}
then the origin of the closed-loop system \eqref{eq:AbstractSys} subject to the ET mechanism \eqref{eq:ETdef} is uniformly globally asymptotically stable (UGAS), that is, there is $\nu \in \KL$ such that the solution $x(t)$ of \eqref{eq:AbstractSys}+\eqref{eq:ETdef} satisfies
\[
\|x(t)\|_H\leq \nu(\|x_0\|_H,t),\quad t\geq 0.
\]
\end{theorem}

\begin{proof}
We first consider the case $x_0\in \operatorname{dom}(A)$. Let $t\in[t_k,t_{k+1})$ with $k\in \mathbb{N}_0$. On this interval, the closed-loop system \eqref{eq:AbstractSys} subject to the ET controller \eqref{eq:ETdef} is given by
\begin{equation}\label{eq:ClosedLoopPf}
\dot{x}(t) = Ax(t)+Bf\left(Fx(t_k) \right), \ t\in [t_k,t_{k+1})
\end{equation}
whence
\begin{equation*}
\begin{array}{lll}
x(t) = T(t-t_k)x(t_k)+\int_{t_k}^tT(t-s)Bf\left(Fx(t_k) \right) \mathrm{d}s.
\end{array}
\end{equation*}
Therefore, by \cite[Theorems 1.2.4 and 6.1.5]{pazy1983semigroups}, and employing induction, we conclude that $x(t_k)\in \operatorname{dom}(A)$ for all $k\in \mathbb{N}_0$ and, furthermore, $x(t)$ is the unique \emph{classical} solution of \eqref{eq:ClosedLoopPf}.

Consider $t\in [t_k,t_k+\tau)$. On this interval, the ET condition is not checked and the system has the dynamics of \eqref{eq:AbstractSys} subject to the sampled-data controller \eqref{eq:SampDat}, meaning that 
\begin{equation}\label{eq:FirstInterval}
\begin{array}{lll}
\left\|x(t) \right\|_H \leq Ge^{-\chi (t-t_k)}\left\|x(t_k) \right\|_H,\quad t\in [t_k,t_k+\tau],
\end{array}
\end{equation}
by Corollary \ref{Cor:StabSamp}. On the interval $[t_k+\tau,t_{k+1})$, we use the fact that $x(t)$ is a classical solution and present \eqref{eq:ClosedLoopPf} as 
\begin{equation}\label{eq:ClosedLoopPf1}
\begin{array}{lll}
\dot{x}(t)= Ax(t) + Bf(Fx(t))+ d(t),\ t\in [t_k+\tau,t_{k+1}),\\
d(t):= Bf(Fx(t_k)) - Bf(Fx(t)).
\end{array}
\end{equation}
As $x$ is a classical solution, it is continuously differentiable, and hence Lipschitz continuous. Thus, $d$ is Lipschitz on $[t_k+\tau,t_{k+1})$ as a composition of Lipschitz maps. 
Employing Assumption \ref{assum:ExistISSLyap} and the ET condition on $[t_k+\tau,t_{k+1})$, we have 
\begin{equation}\label{eq:LyapSecondInterv}
\begin{array}{lll}
D^+V(x(t))&\leq -\alpha\left(V(x(t)) \right)+ \gamma \left(\left\|d(t) \right\|_H \right)\\
&\leq -(1-\sigma)\alpha(V(x(t))).
\end{array}
\end{equation}
Using Remark~\ref{rem:Comparison-principle} with $y(t):=V(x(t))$, we obtain that 
\[
V(x(t)) \leq \kappa(V(x(t_k+\tau)),t-t_k-\tau),\qquad t\in [t_k+\tau,t_{k+1}).
\]
Hence, using Assumption \ref{assum:ExistISSLyap} and \eqref{eq:FirstInterval}
\begin{equation}\label{eq:SecondInterval}
\left\|x(t) \right\|_H\leq \rho(\left\|x(t_k) \right\|_H,t-t_k-\tau),\quad  t\in [t_k+\tau,t_{k+1}].
\end{equation}
Denoting $J_{k+1}:=t_{k+1}-t_k-\tau,\ k\in \mathbb{N}_0$, we see that \linebreak 
$\left\|x(t_k) \right\|_H\leq \mathcal{R}^{(1)}(\left\|x(t_{k-1}) \right\|_H,\{J_k\})$, and by induction we get
\begin{equation}\label{eq:CombInterv}
\begin{array}{lll}
\hspace{-3mm}\left\|x(t_k) \right\|_H&\leq \mathcal{R}^{(k)}\left(\left\| x_0\right\|_H, \left\{ J_j\right\}_{j=1}^k \right)\\
&\overset{\eqref{eq:ETCond}}{\leq} \mu\left(\omega(k)\left\|x_0 \right\|_H, \max_{1\leq j \leq k} J_j\right),\ k\in \mathbb{N}.
\end{array}
\end{equation}

Next, we employ density of $\operatorname{dom}(A)$ in $H$ (see similar arguments in Proposition \ref{prop:ConvISS}) and deduce that \eqref{eq:FirstInterval}, \eqref{eq:SecondInterval}  and \eqref{eq:CombInterv} hold for all $x_0\in H$.

We now show uniform global attractivity of the origin. Let $\eta,\epsilon>0$. Let $N'$ and $J'$ be such that 
\begin{equation}\label{eq:maxcond}
\begin{array}{lll}
&\hspace{-5mm}\max \left\{ G\mu(\omega(N')\eta,0) , \rho\left(\mu(\omega(N')\eta,0),0 \right)\right.\\
&\hspace{8mm}\left.,G \mu(\omega(0)\eta, J'),\rho \left( \mu(\omega(0)\eta, J'),0\right)\right\}\leq \epsilon.
\end{array}
\end{equation}
Let $T=(N'-2)\left(J'+\tau \right)$. We show that $x_0\in H$ and $\left\|x_0 \right\|_H\leq \eta$ implies $\left\|x(t) \right\|_H\leq \epsilon, \ t\geq T$. We consider two cases.

(i) The interval $[0,T]$ contains subintervals $[t_k,t_{k
    +1})$ for some $k\geq N'-1$. Then, any $t\geq T$ belongs to some interval $[t_{\ell},t_{\ell+1})$ for $\ell\geq N'$. If $t\in [t_{\ell},t_{\ell}+\tau)$, then
    \begin{equation*}
    \begin{array}{lll}
    \left\|x(t) \right\|_H\leq Ge^{-\theta (t-t_{\ell})}\left\|x(t_\ell) \right\|_H
    \leq G\mu\left(\omega(N')\eta,0 \right)\leq \epsilon.
    \end{array}
    \end{equation*}
    If $t\in [t_{\ell}+\tau,t_{\ell}+1)$, then 
    \begin{equation*}
    \begin{array}{lll}
    \hspace{-2mm}\left\|x(t) \right\|_H\leq\rho(\left\| x(t_{\ell})\right\|_H,t-t_{\ell}-\tau)\leq \rho\left(\mu\left(\omega(N')\eta,0 \right),0 \right)\leq \epsilon.
    \end{array}
    \end{equation*}
(ii) The interval $[0,T]$ contains $M\leq N'-2$ intervals of the form $[t_{k},t_{k+1})$. Here, we have $\sum_{j=1}^MJ_j+\tau M \leq T$, which implies that there exists some $1\leq j'\leq M$ such that $J_{j'}\geq \frac{T}{M}-\tau \geq \frac{N'-2}{M}\left(J'+\tau \right)-\tau \geq J'$. Let $t\geq T$ such that $t\in [t_{\ell},t_{\ell+1}),\ \ell \geq M+1$. If $t\in [t_{\ell},t_{\ell}+\tau)$, then
    \begin{equation*}
    \begin{array}{lll}
    \left\|x(t) \right\|_H\leq Ge^{-\theta (t-t_{\ell})}\left\|x(t_\ell) \right\|_H\leq \mu\left(\omega(0)\eta,J_{j'} \right)\leq \epsilon.
    \end{array}
    \end{equation*}
    If $t\in[t_{\ell}+\tau,t_{\ell+1})$ then
    \begin{equation*}
    \begin{array}{lll}
    \left\|x(t) \right\|_H&\leq\rho(\left\| x(t_{\ell})\right\|_H,t-t_{\ell}-\tau)\\
    &\leq \rho\left(\mu\left(\omega(0)\eta,J_{j'} \right),0 \right)\leq \epsilon.
    \end{array}
    \end{equation*}
Since the origin is an equilibrium of \eqref{eq:ClosedLoopPf}, by \cite[Theorem 6.1.2]{pazy1983semigroups}, there exists $\delta'>0$ such that $x_0\in H$ with $\left\|x_0 \right\|_H\leq \delta'$ imply $\left\|x(t) \right\|_H\leq \epsilon$ for all $t\in[0,T]$. Furthermore, as the nonlinearity in  \eqref{eq:ClosedLoopPf} is globally Lipschitz, system 
 \eqref{eq:ClosedLoopPf} has bounded reachability sets (see the proof of \cite[Corollary 4.8]{MiW19a}). Hence, \cite[Theorem 6.21]{Mir23}, specialized to systems without inputs, shows that \eqref{eq:ClosedLoopPf} is UGAS.$\hfill \blacksquare$
 
\end{proof}
\begin{remark}
As the proof of Theorem \ref{thm:MainThm} shows, the condition \eqref{eq:ETCond} is used to patch the behavior of the closed-loop system intermittently on intervals of the form $[t_k,t_k+\tau)$ and $[t_k+\tau,t_{k+1})$. The need for patching stems primarily from the absence of converse ISS-Lyapunov theorems for sampled-data systems of the form \eqref{eq:AbstractSys} and \eqref{eq:SampDat}.  The investigation of such ISS-Lyapunov theorems is a topic for future research.
\end{remark}

\begin{corollary}
Let Assumption  \ref{assum:ExistISSLyap} be satisfied with
\begin{equation*}
\begin{array}{lll}
& \alpha(s) = a_3s^p, \quad \psi_i(s) = a_is^p,\ i=1,2,
\end{array}
\end{equation*}
where $p,a_i>0,\ i=1,2,3$. Then, the conditions of Theorem \ref{thm:MainThm} hold provided $\vartheta :=
\left(\frac{a_2}{a_1}\right)^{p^{-1}}G e^{-\theta \tau}<1$. In fact, in this case the system \eqref{eq:AbstractSys}+\eqref{eq:ETdef} is exponentialy stable.
\end{corollary}
\begin{proof}
Let $0<\iota \leq \min \left(\frac{(1-\sigma)a_3}{p},\frac{\ln\left(\vartheta^{-1} \right)}{\tau}\right)$. Since here $\kappa(m,n)=e^{-(1-\sigma)a_3n}m$, a calculation shows that 
\begin{equation*}
\begin{array}{lll}
\mathcal{R}^{(N)}\left(m,\left\{n_j \right\}_{j=1}^N \right) \leq  \vartheta^{N}e^{-\iota\sum_{j=1}^Nn_j}m.
\end{array}
\end{equation*}  
Employing this is \eqref{eq:CombInterv}, we obtain
\begin{equation*}
\begin{array}{lll}
&\left\|x(t_k) \right\|_H  \leq \vartheta^{N}e^{-\iota\sum_{j=1}^kJ_j}\left\|x_0 \right\|_H \\
&\hspace{7mm}= \vartheta^k e^{-\iota \sum_{j=1}^{k}\left(t_j-t_{j-1}-\tau \right)}\left\|x_0 \right\|_H
\leq e^{-\iota t_k}\left\|x_0 \right\|_H.
\end{array}
\end{equation*}
Employing the latter in \eqref{eq:FirstInterval} and \eqref{eq:SecondInterval} then finishes the proof of the corollary.
\end{proof}

\section{Example}
Motivated by \cite{MPW21}, we consider the stabilization problem of a linear reaction-diffusion equation via distributed control $u:\R_+\to\R^m$. Let $L>0$. We are given $m$ functions
$b_k: [0,L] \to \R$, $k=1,\ldots,m$ which describe at which places the control input $u_k\in\R$
is acting. The function $c$ models the place-dependent reaction rate. The
system model is then  
\begin{equation}\label{closed-loop:sat}
\begin{split}
& w_t(t,x)=w_{xx}(t,x)+c(x)w(t,x)\\
&\qquad\qquad\quad + \sum_{k=1}^m b_k(x) f(u_k(t)),  \ t>0, \; x \in (0,L), \\
& w(t,0)=w(t,L)=0, \quad t>0, \\
& w(0,x)=w^0(x) , \quad x \in (0,L).
\end{split}
\end{equation}
We assume that the state space of this system is $X:=L_2(0,L)$ and that
$c,b_k \in X$, $k=1,\ldots,m$. 
Finally, $f$ is a globally Lipschitz map on $\R$. 

For $k \in \N$ and $L>0$, $H^{k}(0,L)$ denotes the
Sobolev space of functions from the space $L_2(0,L)$, which have weak
derivatives of order $\leq k$, all of which belong to
$L_2(0,L)$. $H^{k}_0(0,L)$ is the closure of $C_0^k(0,L)$ 
(the $k$-times
continuously differentiable functions with compact support in $(0,L)$)
 in
the norm of $H^{k}(0,L)$.

Define $\operatorname{dom}(A)=H^2(0,L)\cap H^1_0(0,L)$ and an operator
\begin{equation}
A=\partial_{xx}+c(\cdot)\mathrm{id}: \quad \operatorname{dom}(A)\subset X \to X.
\end{equation}
We rewrite \eqref{closed-loop:sat} in an abstract form
\begin{equation}
\label{newzero}
w_t(t,\cdot)=Aw(t,\cdot)+ \sum_{k=1}^m  b_k f(u_k(t)).
\end{equation}

We note that $A$ is selfadjoint and has compact resolvent (see \cite{MPW21} for all details here and in the following).  
Hence, the spectrum of $A$ consists of only isolated eigenvalues with finite multiplicity, 
see \cite[Theorem~III.6.29]{Kat95}. 
Furthermore, there exists a Hilbert basis $\{e_j\}_{j\geq 1}$ of $X$ consisting of eigenfunctions of $A$, associated with the sequence of  eigenvalues $\{\lambda_j\}_{j\geq 1}$. Note that $e_j(\cdot)\in \operatorname{dom}(A)$ for every $j\geq 1$ and
\begin{equation*}
-\infty<\cdots<\lambda_j<\cdots<\lambda_1\quad \textrm{and}\quad \lambda_j\underset{j\rightarrow +\infty}{\longrightarrow}-\infty.     
\end{equation*}

We consider (mild) solutions of the system \eqref{closed-loop:sat} (see \cite[Section 5.1]{CuZ20}), which exist and are unique for any initial condition in $X$ and for any $\{u_k\} \subset L_{1,loc}(\R_+)$.

We expand a solution $w(t,\cdot)\in \operatorname{dom}(A)$ of \eqref{newzero} as the coefficients $b_k$ into series in $e_j(\cdot)$, convergent in
$H_0^1(0,L)$,
\begin{equation*}
    \begin{aligned}
w(t,\cdot)&=\sum_{j=1}^{\infty}w_j(t)e_j(\cdot),
\ \ b_k(\cdot) = \sum_{j=1}^{\infty}b_{jk} e_j(\cdot),\ j\in\N.
    \end{aligned}
\end{equation*}
Now, we represent \eqref{newzero} as
 the infinite-dimensional system
\begin{align}
\label{sys-dim-infinie}
\dot{w}_j(t) 
&=   \lambda_jw_j(t)+  \mathbf{b}_{j} \cdot  f(u(t)) ,\qquad
j\in\N,
\end{align}
where $w(t):=\{w_j(t)\}_{j\in\N} \in \ell_2$, where 
\[
\ell_2:=\Big\{(x_k)_{k\in\N}\in\R^{\N}\ : \ \sum_{k=1}^\infty|x_k|^2 <\infty\Big\}.
\]
The scalar product in $\R^m$ we denote by \q{$\cdot$}, $f(u(t)) \in \R^m$ is the vector with entries $f(u_k(t))$
and $\mathbf{b}_{j}$ is the row vector with entries $b_{jk}$, $k=1,\ldots,m$.

Let $n\in\N$ be the number of nonnegative eigenvalues of $A$
and let $\eta>0$ be such that
\begin{equation}\label{refeta}
\forall j>n:\quad \lambda_j<-\eta<0.
\end{equation}

With the matrix notation
\begin{equation*}
z:= \small \begin{pmatrix} w_1 \\ \vdots \\ w_n
\end{pmatrix}, \ 
\mathbf{A}:= \small \begin{pmatrix}
 \lambda_1 &\! \cdots \! &    0           \\
 \vdots         & \!\ddots \! &   \vdots       \\
  0              &\! \cdots \!& \lambda_n
\end{pmatrix}, \  
\mathbf{B}:= \small\begin{pmatrix}  b_{11} &\!\! \cdots \!\! & b_{1m} \\ \vdots &&
    \vdots\\ b_{n1} &\!\! \cdots\!\! & b_{nm}  \end{pmatrix} 
\end{equation*}
the $n$ first equations of 
\eqref{sys-dim-infinie}
form the unstable finite-dimensional control system 
\begin{equation}\label{systfini}
\dot{z}(t)=\mathbf{A} z(t) + \mathbf{B}f(u(t)).
\end{equation}

Now, the application of our results implies that 
\begin{prop}
Let \eqref{systfini} with $f:=\id$ be stabilizable by some feedback $u(t) = Kz(t)$.
Let $f$ satisfy Assumption~\ref{assum:StabilizingFeed}
with $\theta_f$ satisfying \eqref{eq:PowerStabCond}.
Then there is $\tau^*>0$ and a sample-data controller achieving the stabilization of 
\eqref{closed-loop:sat}.
\end{prop}

\begin{proof}
As \eqref{systfini} with $f:=\id$ is stabilizable by feedback $u(t) = Kz(t)$, for a certain matrix $K$, the results of \cite{MPW21} show that the same feedback exponentially stabilizes \eqref{sys-dim-infinie} with $f:=\id$ (as it is a cascade coupling of an exponentially stable system with an input-to-state stable system). 
As \eqref{sys-dim-infinie} is a reformulation of \eqref{closed-loop:sat}, then also 
\eqref{closed-loop:sat} will be exponentially stabilized by this feedback, which is of finite rank and hence a compact operator.
This, the assumptions of Proposition~\ref{Prop:NonlinPowStab} are fulfilled, and it ensures that the sampled version of $u(t)=Kz(t)$ stabilizes \eqref{closed-loop:sat}. \hfill $\blacksquare$
\end{proof}

\begin{remark}
Having developed a 
controller for \eqref{closed-loop:sat}, one can approach developing a corresponding event-triggered mechanism using the outcomes of Section~\ref{sec:Event-based stabilization}. 
To this end, one can use 1-homogeneous ISS Lyapunov functions for the disturbed closed-loop system developed in Proposition~\ref{prop:ConvISS}, or exploit quadratic ISS Lyapunov functions following \cite{MPW21}.
Due to page limits, we skip this analysis.
\end{remark}

\bibliographystyle{abbrv}
\bibliography{IEEEabrv,BibliographyKSE,MyPublications,Mir_LitList_NoMir}

\begin{thebibliography}{10}

\bibitem{BaC16}
G.~Bastin and J.-M. Coron.
\newblock {\em Stability and Boundary Stabilization of 1-{D} Hyperbolic
  Systems}.
\newblock Springer, 2016.

\bibitem{BoH14}
D.~N. Borgers and W.~M. Heemels.
\newblock Event-separation properties of event-triggered control systems.
\newblock {\em IEEE Transactions on Automatic Control}, 59(10):2644--2656,
  2014.

\bibitem{Con90}
J.~B. Conway.
\newblock {\em A Course in Functional Analysis}.
\newblock Springer, New York, 1990.

\bibitem{CuZ20}
R.~Curtain and H.~Zwart.
\newblock {\em Introduction to Infinite-Dimensional Systems Theory: {A}
  State-Space Approach}.
\newblock Springer, 2020.

\bibitem{EGM18}
N.~Espitia, A.~Girard, N.~Marchand, and C.~Prieur.
\newblock Event-based boundary control of a linear $2\times 2$ hyperbolic
  system via backstepping approach.
\newblock {\em IEEE Transactions on Automatic Control}, 63(8):2686--2693, 2018.

\bibitem{EKK21}
N.~Espitia, I.~Karafyllis, and M.~Krstic.
\newblock Event-triggered boundary control of constant-parameter
  reaction--diffusion {PDE}s: {A} small-gain approach.
\newblock {\em Automatica}, 128:109562, 2021.

\bibitem{Gib80}
J.~Gibson.
\newblock A note on stabilization of infinite dimensional linear oscillators by
  compact linear feedback.
\newblock {\em SIAM Journal on Control and Optimization}, 18(3):311--316, 1980.

\bibitem{gilmore2020infinite1121}
M.~E. Gilmore, C.~Guiver, and H.~Logemann.
\newblock Infinite-dimensional {L}ur’e systems with almost periodic forcing.
\newblock {\em Mathematics of Control, Signals, and Systems}, 32:327--360,
  2020.

\bibitem{gilmore2020stability111}
M.~E. Gilmore, C.~Guiver, and H.~Logemann.
\newblock Stability and convergence properties of forced infinite-dimensional
  discrete-time {L}ur'e systems.
\newblock {\em International journal of control}, 93(12):3026--3049, 2020.

\bibitem{guiver2019infinite111}
C.~Guiver, H.~Logemann, and M.~R. Opmeer.
\newblock Infinite-dimensional {L}ur'e systems: Input-to-state stability and
  convergence properties.
\newblock {\em SIAM Journal on Control and Optimization}, 57(1):334--365, 2019.

\bibitem{JMP20}
B.~Jacob, A.~Mironchenko, J.~R. Partington, and F.~Wirth.
\newblock Noncoercive {L}yapunov functions for input-to-state stability of
  infinite-dimensional systems.
\newblock {\em SIAM Journal on Control and Optimization}, 58(5):2952--2978,
  2020.

\bibitem{Kat95}
T.~Kato.
\newblock {\em Perturbation Theory for Linear Operators}.
\newblock Springer, Berlin, 1995.
\newblock Reprint of the 1980 edition.

\bibitem{KFS21}
R.~Katz, E.~Fridman, and A.~Selivanov.
\newblock Boundary delayed observer-controller design for reaction--diffusion
  systems.
\newblock {\em IEEE Transactions on Automatic Control}, 66(1):275--282, 2021.

\bibitem{KBT22}
F.~Koudohode, L.~Baudouin, and S.~Tarbouriech.
\newblock Event-based control of a damped linear wave equation.
\newblock {\em Automatica}, 146:110627, 2022.

\bibitem{logemann2014infinite}
H.~Logemann.
\newblock Infinite-dimensional {L}ur’e systems: the circle criterion,
  input-to-state stability and the converging-input-converging-state property.
\newblock In {\em Proceedings of the 21st international symposium on
  mathematical theory of networks and systems, Groningen, The Netherlands},
  pages 1624--1627, 2014.

\bibitem{LRT03}
H.~Logemann, R.~Rebarber, and S.~Townley.
\newblock Stability of infinite-dimensional sampled-data systems.
\newblock {\em Transactions of the American Mathematical Society},
  355(8):3301--3328, 2003.

\bibitem{Mir23}
A.~Mironchenko.
\newblock {\em Input-to-State Stability}.
\newblock Springer, 2023.

\bibitem{MiP20}
A.~Mironchenko and C.~Prieur.
\newblock Input-to-state stability of infinite-dimensional systems: Recent
  results and open questions.
\newblock {\em SIAM Review}, 62(3):529--614, 2020.

\bibitem{MPW21}
A.~Mironchenko, C.~Prieur, and F.~Wirth.
\newblock Local stabilization of an unstable parabolic equation via saturated
  controls.
\newblock {\em IEEE Transactions on Automatic Control}, 66(5):2162--2176, 2021.

\bibitem{MiS25}
A.~Mironchenko and F.~Schwenninger.
\newblock Coercive quadratic converse {ISS} {L}yapunov theorems for linear
  analytic systems.
\newblock {\em Submitted to Mathematics of Control, Signals, and Systems},
  2025.

\bibitem{MiW17c}
A.~Mironchenko and F.~Wirth.
\newblock Lyapunov characterization of input-to-state stability for semilinear
  control systems over {B}anach spaces.
\newblock {\em Systems \& Control Letters}, 119:64--70, 2018.

\bibitem{MiW19a}
A.~Mironchenko and F.~Wirth.
\newblock Non-coercive {L}yapunov functions for infinite-dimensional systems.
\newblock {\em Journal of Differential Equations}, 105:7038--7072, 2019.

\bibitem{pazy1983semigroups}
A.~Pazy.
\newblock {\em Semigroups of Linear Operators and Applications to Partial
  Differential Equations}, volume~44.
\newblock Springer New York, 1983.

\bibitem{PTN15}
R.~Postoyan, P.~Tabuada, D.~Ne{\v{s}}i{\'c}, and A.~Anta.
\newblock A framework for the event-triggered stabilization of nonlinear
  systems.
\newblock {\em IEEE Transactions on Automatic Control}, 60(4):982--996, 2015.

\bibitem{RDE21}
B.~Rathnayake, M.~Diagne, N.~Espitia, and I.~Karafyllis.
\newblock Observer-based event-triggered boundary control of a class of
  reaction--diffusion pdes.
\newblock {\em IEEE Transactions on Automatic Control}, 67(6):2905--2917, 2021.

\bibitem{selivanov2015event11}
A.~Selivanov and E.~Fridman.
\newblock Event-triggered \uppercase{H}$_{\infty}$ control: A switching
  approach.
\newblock {\em IEEE Transactions on Automatic Control}, 61(10):3221--3226,
  2015.

\bibitem{SoW95}
E.~D. Sontag and Y.~Wang.
\newblock On characterizations of the input-to-state stability property.
\newblock {\em Systems \& Control Letters}, 24(5):351--359, 1995.

\bibitem{Tab07}
P.~Tabuada.
\newblock Event-triggered real-time scheduling of stabilizing control tasks.
\newblock {\em IEEE Transactions on Automatic Control}, 52(9):1680--1685, 2007.

\bibitem{ZhZ18}
J.~Zheng and G.~Zhu.
\newblock Input-to-state stability with respect to boundary disturbances for a
  class of semi-linear parabolic equations.
\newblock {\em Automatica}, 97:271--277, 2018.

\end{thebibliography}

\end{document}